\documentclass{article}

\usepackage{authblk}
\usepackage[inline]{enumitem}
\usepackage{color}
\usepackage{graphicx}
\usepackage{tikz}
\usepackage[numbers]{natbib}
\usepackage{amsthm,amsmath,amssymb}
\newtheorem{theorem}{Theorem}
\newtheorem{lem}{Lemma}

\newtheorem{defn}{Definition}
\newtheorem{prop}{Proposition}
\newtheorem{cor}{Corollary}

\def\eqdef{\stackrel{\rm def}{=}}
\usepackage{mathtools,mathabx}
\usepackage{times}
\usepackage{color}
\usepackage{multicol}
\usepackage{bm}

\usepackage{ebproof}
\usepackage{cancel}

\newcommand{\forrefs}[1]{}

\newcommand{\simrel}{\mathrel E}
\newcommand{\simu}{\rightbarharpoon}
\newcommand{\remc}[1]{\ominus #1}
\newcommand{\CIcon }{c}
\newcommand{\ptype}{\type\infty}
\newcommand{\Sim}[1]{{\rm Sim}({#1})}
\newcommand{\CMod}{\model_\CIcon }
\newcommand{\model}{{\cl M}}
\newcommand{\tnext}{\circ}
\newcommand{\peq}{\preccurlyeq}
\newcommand{\ubox}{\Box}
\newcommand{\diam}{\lozenge}

\newcommand{\itle}{{\sf ITL^e}}
\newcommand{\itlp}{{\sf ITL^p}}
\newcommand{\ignore}[1]{}
\newcommand{\peqT}{\peq_{T}}
\newcommand{\subT}{\subseteq_{T}}
\newcommand{\ST}{\mathrel S_T}

\newcommand{\logbasic}{{\sf ITL}^0_\diam}

\newcommand{\type}[1]{T_{ #1 }}

\newcommand{\landi}{\cl L_\diam}

\newcommand{\unp}[1]{J_{#1}}

\newcommand\cqm[1]{\cl J_{ #1 }}
\newcommand{\irr}[1]{\cl I_{ #1 }}

\def\<{\left (}

\def\>{\right )}
\def\({\left (}
\def\){\right )}

\def\cbra{\left \{}
\def\cket{\right \}}

\def\eqdef{\stackrel{\rm def}{=}}

\def\tnext{\ovoid}

\def\<{\left (}

\def\>{\right )}

\def\cbra{\left \{}
\def\cket{\right \}}
\newcommand{\lang}{\cl L}

\newcommand{\fr}{\bm }
\newcommand{\cl}{\mathcal}
\newcommand{\imp}{\mathop \to}
\newcommand{\seq}{\succcurlyeq}

\newcommand{\Root}[1]{r_{#1}}

\newcommand{\fw}{{{{\fr w}}}}
\newcommand{\fv}{{{{{\fr v}}}}}

\newcommand{\redu}{\mathrel\unlhd}

\def\M{{\cal M}}

\def\iltl{{\sf ITL^e}}

\makeatletter
\newcommand{\manuallabel}[2]{\def\@currentlabel{#2}\label{#1}}
\makeatother

\usepackage{hyperref}

\title{An intuitionistic axiomatization of `eventually'}
\author[1]{Mart\'in Di\'eguez \footnote{\href{mailto:martin.dieguez@enib.fr}{\tt martin.dieguez@enib.fr}}}
\author[2]{David Fern\'andez-Duque \footnote{\href{mailto:david.fernandezduque@ugent.be}{\tt david.fernandezduque@ugent.be}}}
\affil[1]{CERV, ENIB,LAB-STICC. Brest, France}
\affil[2]{Department of Mathematics, Ghent University. Gent, Belgium}

\begin{document}
\maketitle

\begin{abstract}
Boudou and the authors have recently introduced the intuitionistic temporal logic $\itle$ and shown it to be decidable. 
In this article we show that the `henceforth'-free fragment of this logic is complete for the class of {non-deterministic quasimodels} introduced by Fern\'andez-Duque \cite{FernandezITLc}. From this and results of Boudou, Romero and the authors \cite{BoudouLICS}, we conclude that this fragment is also complete for the standard semantics of $\itle$ over the class of {expanding posets.}
\end{abstract}
	

\section{Introduction}

Intuitionistic logic is the basis for constructive reasoning and temporal logics are an important tool for reasoning about dynamic processes.
One would expect that a combination of the two would yield a powerful framework in which to model phenomena involving both computation and time, an idea explored by Davies~\cite{Davies96} and Maier ~\cite{Maier2004Heyting}.
This is not the only potential application of such a logic: in view of the topological interpretation of the intuitionistic implication, one may instead use it to model {\em space} and time \cite{FernandezITLc}.
This makes it important to study these logics, which in particular did not previously enjoy a complete axiomatization in the presence of `infinitary' tenses. Our goal in this paper is to present such an axiomatization for `next' and `eventually'.

\subsection{State-of-the-art}

There are several (poly)modal logics which may be used to model time, and some have already been studied in an intuitionistic setting, e.g.~tense logics by Davoren \cite{Davoren2009} and propositional dynamic logic with iteration by Nishimura \cite{NishimuraConstructivePDL}.
Here we are specifically concerned with intuitionistic analogues of discrete-time linear temporal logic.
Versions of such a logic in finite time have been studied by Kojima and Igarashi \cite{KojimaNext} and Kamide and Wansing \cite{KamideBounded}.
Nevertheless, logics over infinite time have proven to be rather difficult to understand, in no small part due to their similarity to intuitinoistic modal logics such as $\sf IS4$, whose decidablitiy has long remained open \cite{Simpson94}.

In recent times, Balbiani, Boudou and the authors have made some advances in this direction, showing that the intermediate logic of temporal here-and-there is decidable and enjoys a natural axiomatization \cite{BalbianiDieguezJelia} and identifying two conservative temporal extensions of intuitionistic logic, denoted $\itle$ and $\itlp$ (see \S\ref{SecSemantics}). These logics are based on the temporal language with $\tnext$ (`next'), $\diam$ (`eventually') and $\ubox$ (`henceforth'); note that unlike in the classical case, the latter two are not inter-definable \cite{IMLA}.
Both logics are given semantically and interpreted over the class of {\em dynamic posets,} structures of the form $\mathcal F = (W,{\peq},S)$ where $\peq$ is a partial order on $W$ used to interpret implication and $S\colon W\to W$ is used to interpret tenses. If $w\peq v$ implies that $S(w) \peq S(v)$ we say that $\mathcal F$ is an {\em expanding poset;} $\itle$ is then defined to be the set of valid formulas for the class of expanding posets, while $\itlp$ is the logic of {\em peristent posets,} where $\mathcal F$ has the additional {\em backward confluence} condition stating that if $v \seq S(w)$, then there is $u\seq w$ such that $S(u) = v$.

Unlike $\itle$, the logic $\itlp$ satisfies the familiar Fischer Servi axioms \cite{FS84}; nevertheless, $\itle$ has some technical advantages. We have shown that $\itle$ has the small model property while $\itlp$ does not \cite{BoudouCSL}; this implies that $\itle$ is decidable. It is currently unknown if $\itlp$ is even axiomatizable, and in fact its modal cousin ${\sf LTL} \times {\sf S4}$ is not computably enumerable \cite{pml}. On the other hand, while $\itle$ is axiomatizable in principle, the decision procedure we currently know uses model-theoretic techniques and does not suggest a natural axiomatization.

In \cite{BoudouLICS} we laid the groundwork for an axiomatic approach to intuitionistic temporal logics, identifying a family of natural axiom systems that were sound for different classes of structures, including a `minimal' logic ${\sf ITL}^0$ based on a standard axiomatization for $\sf LTL$.
There we consider a wider class of models based on topological semantics and show that ${\sf ITL}^0$ is sound for these semantics, while
\begin{multicols}2
\begin{enumerate}[label=(\alph*)]

\item\label{CDAx} $\ubox(p\vee q) \rightarrow \diam p \vee \ubox q$

\item\label{BIAx} $\ubox (\tnext p \to p) \wedge \ubox(p\vee q)  \rightarrow p \vee \ubox q$

\end{enumerate}
\end{multicols}
\noindent are Kripke-, but not topologically, valid, from which it follows that these principles are not derivable in ${\sf ITL}^0$.

On the other hand, it is also shown in \cite{BoudouLICS} that for $\varphi \in \cl L_\diam$, the following are equivalent:
\begin{enumerate}

\item $\varphi$ is topologically valid,

\item $\varphi$ is valid over the class of expanding posets,

\item $\varphi$ is valid over the class of finite quasimodels.

\end{enumerate}
Quasimodels are discussed in \S\ref{SecNDQ} and are the basis of the completeness for dynamic topological logic presented in \cite{dtlaxiom}, which works for topological, but not Kripke, semantics. This suggests that similar techniques could be employed to give a completeness proof for a natural logic over the $\ubox$-free fragment, but not necessarily over the full temporal language; in fact, we do not currently have a useful notion of quasimodel in the presence of $\ubox$. Moreover, \ref{CDAx} and \ref{BIAx} are not valid in most intuitionistic modal logics, and there is little reason at this point to suspect that no other independent validities are yet to be discovered.
For this reason, in this manuscript we restrict our attention to the $\ubox$-free fragment of the temporal language, which we denote $\cl L_\diam$, and we will work with the logic $\logbasic$, a $\ubox$-free version of ${\sf ITL}^0$.

\subsection{Our main result}

The goal of this article is to prove that $\logbasic$ is complete for the class of non-deterministic quasimodels (Theorem \ref{theocomp}).
The completeness proof follows the general scheme of that for linear temporal logic \cite{temporal}: a set of `local states', which we will call {\em moments,} is defined, where a moment is a representation of a potential point in a model (or, in our case, a quasimodel). To each moment $w$ one then assigns a characteristic formula $\chi(w)$ in such a way that $\chi(w)$ is consistent if and only if $w$ can be included in a model, from which completeness can readily be deduced.

In the $\sf LTL$ setting, a moment is simply a maximal consistent subset of a suitable finite set $\Sigma$ of formulas.
For us a moment is instead a finite labelled tree, and the formula $\chi(w)$ must characterize $w$ up to {\em simulation;} for this reason we will henceforth write $\Sim{w}$ instead of $\chi(w)$.
The required formulas $\Sim{w}$ can readily be constructed in $\landi$ (Proposition \ref{simulability}).

Note that it is {\em failure} of $\Sim{w}$ that characterizes the property of simulating $w$, hence the {\em possible} states will be those moments $w$ such that $\Sim{w}$ is unprovable.
The set of possible moments will form a quasimodel falsifying a given unprovable formula $\varphi$ (Corollary \ref{laststretch}). Thus any unprovable formula is falsifiable, and Theorem \ref{theocomp} follows. We then conclude from Theorem \ref{TheoITLc}, proven in Boudou et al.~\cite{BoudouLICS}, that any unprovable formula is also falsifiable in an expanding poset (Corollary \ref{corcomplete}).

\subsection*{Layout} Section \ref{SecBasic} introduces the syntax and semantics of $\itle$, and Section \ref{SecNDQ} discusses labelled structures, which generalize both models and quasimodels. 
Section \ref{secCanMod} discusses the canonical model, which properly speaking is a deterministic weak quasimodel.
Section \ref{SecSim} reviews simulations and dynamic simulations, including their definability in the intuitionistic language.
Section \ref{seccan} constructs the initial quasimodel and establishes its basic properties, but the fact that it is in fact a quasimodel is proven only in Section \ref{secOmSens} where it is shown that the quasimodel is {\em $\omega$-sensible,} i.e.~it satisfies the required condition to interpret $\diam$. The completeness of $\logbasic$ follows immediately from this fact.

In Appendix \ref{apCanMod} we include the proof that the canonical model is a weak quasimodel, Appendix \ref{secsubchar} gives an explicit construction of simulation formulas and Appendix \ref{secFinFrm} reviews the construction of the initial weak quasimodel from \cite{FernandezITLc}. \forrefs{Finally, Appendix \ref{appReferees} discusses the status of unpublished work for the referees' benefit.}

\section{Syntax and semantics}\label{SecBasic}

Fix a countably infinite set $\mathbb P$ of `propositional variables'. The language $\cl L$ of intuitionistic (linear) temporal logic $\sf ITL$ is given by the grammar
\[ \bot \  | \   p  \ |  \ \varphi\wedge\psi \  |  \ \varphi\vee\psi  \ |  \ \varphi\to\psi  \ |  \ \tnext\varphi \  | \  \diam\varphi \  |  \ \ubox\varphi, \]
where $p\in \mathbb P$. As usual, we use $\neg\varphi$ as a shorthand for $\varphi\to \bot$ and $\varphi \leftrightarrow \psi$ as a shorthand for $(\varphi \to \psi) \wedge (\psi \to \varphi)$. We read $\tnext$ as `next', $\diam$ as `eventually', and $\ubox$ as `henceforth'.
Given any formula $\varphi$,
we denote the set of subformulas of $\varphi$ by ${\mathrm{sub}}(\varphi)$.
We will work mainly in the language $\cl L_{\diam}$, defined as the sublanguage of $\cl L$ without the modality $\ubox$, although the full language will be discussed occasionally.

\subsection{Semantics}\label{SecSemantics}

Formulas of $\lang$ are interpreted over expanding posets. An {\em expanding poset} is a tuple $\mathcal D = (|\cl D|,{\peq}_\cl D,S_\cl D)$,
where 
$|\cl D|$ is a non-empty set of moments,
$\peq_\cl D$ is a partial order over $|\cl D|$,
and
$S_\cl D$ is a function from $|\cl D|$ to $|\cl D|$ satisfying the {\em forward confluence} condition that for all $ w, v \in |\cl D|, $ if $ w \peq_\cl D v $ then $ S_\cl D(w) \peq S_\cl D(v).$
We will omit the subindices in $\peq_\cl D$, $S_\cl D$ when $\cl D$ is clear from context and write $v\prec w$ if $v\peq w$ and $v\not = w$.
An {\em intuitionistic dynamic model,} or simply {\em model,} is a tuple $\model = (|\model|,\peq_\model,S_\model,V_\model)$
consisting of an expanding poset equipped
with a valuation function $V$ from $|\cl M|$ to sets of propositional variables that is {$\peq$-monotone,} in the sense that
for all $ w, v \in W, $ if $ w \peq   v $ then $ V  (w) \subseteq V  (v).$ In the standard way, we define $S^0(w) = w$ and,
for all $k > 0$, $S^k(w) = S\left(S^{k-1}(w)\right)$.
 Then we define the satisfaction relation $\models$ inductively by:

\noindent\begin{enumerate}
	\item $\model, w \models p $  if $p \in V(w) $;
	\item $\model, w \not\models \bot$;
	\item $\model, w \models \varphi\wedge \psi$  if $  \model, w \models \varphi $ and $\model, w \models \psi$;
	\item $\model, w \models \varphi \vee \psi $  if  $ \model, w \models \varphi $ or $\model, w \models \psi$;
	\item $\model, w \models \tnext \varphi $   if $ \model, S(w) \models \varphi $;
		\item  $\model, w \models \varphi \imp \psi $ if $\forall v \seq w $, if $\model, v \models \varphi$ then $\model, v \models \psi $;
	\item $\model, w \models \diam \varphi $  if there exists $ k $ such that $ \model, S^k(w) \models \varphi$;
 \item $\model, w \models \ubox \varphi $  if for all $ k $, $ \model, S^k(w) \models \varphi$.
\end{enumerate}
As usual, a formula $\varphi$ is {\em valid over a class of models $\Omega$} if, for every world $w$ of every model $\model\in \Omega$, $\model,w\models\varphi$.
The set of valid formulas over an arbitrary expanding poset will be called $\iltl$, or {\em expanding intuitionistic temporal logic;} the terminology was coined in \cite{BoudouCSL} and is a reference to the closely-related {\em expanding products} of modal logics \cite{pml}. The main result of \cite{BoudouCSL} is the following.

\begin{theorem}\label{ThoDecid}
$\iltl$ is decidable.
\end{theorem}

Nevertheless, Theorem \ref{ThoDecid} is proved using purely model-theoretic techniques that do not suggest an axiomatization in an obvious way.
In \cite{BoudouLICS} we introduced the axiomatic system ${\sf ITL}^0$, inspired by standard axiomatizations for $\sf LTL$. As we will see, adapting this system to $\cl L_\diam$ yields a sound and complete deductive calculus for the class of expanding posets.

\subsection{The axiomatization}

Our axiomatization obtained from propositional intuitionistic logic \cite{MintsInt} by adding standard axioms and inference rules of $\sf LTL$ \cite{temporal}, although modified to use $\diam$ instead of $\ubox$.
To be precise, the logic $\logbasic$ is the least set of $\cl L_\diam$-formulas closed under the following axiom schemes and rules:
\medskip

\begin{multicols}2
\begin{enumerate}[label=(A\arabic*)]
\item\label{ax01Taut} All intuitionistic tautologies
\item\label{ax02Bot} $\neg \tnext \bot$
\item\label{ax03NexWedge} $ \tnext \varphi \wedge\tnext \psi  \rightarrow \tnext \left( \varphi \wedge \psi \right)$
\item\label{ax04NexVee} $ \tnext \left( \varphi \vee \psi \right) \rightarrow  \tnext \varphi \vee\tnext \psi $
\item\label{ax05KNext} $\tnext\left( \varphi \rightarrow \psi \right) \rightarrow \left(\tnext\varphi \rightarrow \tnext\psi\right)$
\item\label{ax10DiamFix} $\varphi \vee \tnext \diam \varphi \to \diam \varphi$
\end{enumerate}
\end{multicols}

\begin{multicols}2
\begin{enumerate}[label=({R\arabic*})]

\item\label{ax13MP} $\displaystyle\frac {\varphi \ \ \ \varphi\to \psi }\psi$
\item\label{ax14NecCirc} $\displaystyle\frac \varphi {\tnext\varphi}$
\item\label{ax11:dist} $\displaystyle\frac{ \varphi \rightarrow  \psi } { \diam \varphi \rightarrow \diam \psi }$
\item\label{ax12:ind:2} $\displaystyle\frac{ \tnext \varphi \to \varphi} { \diam \varphi \rightarrow \varphi } $

\end{enumerate}
\end{multicols}

The axioms \ref{ax02Bot}-\ref{ax05KNext} are standard for a functional modality.  Axiom \ref{ax10DiamFix} is the dual of $\ubox \varphi \rightarrow \varphi \wedge \tnext \ubox \varphi$.
The rule \ref{ax11:dist} replaces the dual K-axiom $\ubox(\varphi \to \psi)\to(\diam \varphi \to \diam \psi)$, while \ref{ax12:ind:2} is dual to the induction rule $ \frac{ \varphi \to \tnext  \varphi} { \varphi \rightarrow \ubox \varphi } $.
As we show next, we can also derive the converses of some of these axioms. Below, for a set of formulas $\Gamma$ we define $\tnext \Gamma = \{\tnext\varphi : \varphi \in \Gamma\}$, and empty conjunctions and disjunctions are defined by $\bigwedge\varnothing =\top$ and $\bigvee \varnothing = \bot$.

\begin{lem}\label{lemmReverse}
Let $\varphi \in \landi$ and $\Gamma\subseteq \landi$ be finite. Then, the following are derivable in $\logbasic$:
\begin{enumerate}

\item $\tnext \bigwedge \Gamma \leftrightarrow \bigwedge \tnext \Gamma$

\item $\tnext \bigvee \Gamma \leftrightarrow \bigvee \tnext \Gamma$

\item\label{itReverseDiam} $\diam \varphi \to \varphi \vee \tnext \diam \varphi$.
\end{enumerate}

\end{lem}

\proof
For the first two claims, one direction is obtained from repeated use of axioms \ref{ax03NexWedge} or \ref{ax04NexVee} and the other is proven using \ref{ax14NecCirc} and \ref{ax05KNext}; note that the second claim requires \ref{ax02Bot} to treat the case when $\Gamma = \varnothing$. Details are left to the reader.

For the third claim, reasoning within $\logbasic$, note that $\varphi \to \diam \varphi$ holds by \ref{ax10DiamFix} and propositional reasoning, hence $\tnext\varphi \to \tnext \diam \varphi$ by \ref{ax14NecCirc}, \ref{ax05KNext} and \ref{ax13MP}.
By similar reasoning, $\tnext \tnext \diam \varphi \to \tnext \diam \varphi$ holds, hence so does $\tnext \varphi \vee \tnext \tnext \diam \varphi \to \tnext \diam \varphi$. Using \ref{ax04NexVee} and some propositional reasoning we obtain $\tnext(\varphi \vee \tnext \diam \varphi) \to \varphi \vee \tnext \diam \varphi$.
But then, by \ref{ax12:ind:2}, $\diam(\varphi \vee \tnext \diam \varphi) \to \varphi \vee \tnext \diam \varphi$; since $\diam\varphi \to \diam (\varphi \vee \tnext \diam \varphi)$ can be proven using \ref{ax11:dist}, we obtain $\diam\varphi \to \varphi \vee \tnext \diam \varphi$, as needed.
\endproof

For purposes of this discussion, a {\em logic} may be any set $\Lambda \subseteq \cl L$, and we may write $\Lambda \vdash \varphi$ instead of $\varphi \in \Lambda$. Then, $\Lambda$ is {\em sound} for a class of structures $\Omega$ if, whenever $\Lambda \vdash \varphi$, it follows that $\Omega \models \varphi$.
The following is essentially proven in \cite{BoudouLICS}:
   
\begin{theorem}\label{ThmSoundZero}
  $\logbasic$ is sound for the class of expanding posets.
  \end{theorem}

Note however that a few of the axioms and rules have been modified to fall within $\cl L_\diam$, but these modifications are innocuous and their correctness may be readily checked by the reader.
We remark that, in contrast to Lemma \ref{lemmReverse}, $(\tnext p\to \tnext q)\to \tnext(p\to q)$ is not valid \cite{IMLA}, hence by Theorem \ref{ThmSoundZero}, it is not derivable.

\section{Labelled structures}\label{SecNDQ}

The central ingredient of our completeness proof is given by non-deterministic {\em quasimodels,} introduced by Fern\'andez-Duque in the context of dynamic topological logic \cite{FernandezNonDeterministic} and later adapted to intuitionistic temporal logic \cite{FernandezITLc}.

\subsection{Two-sided types}

Quasimodels are structures whose worlds are labelled by types, as defined below.
More specifically, following \cite{BoudouLICS}, our quasimodels will be based on two-sided types. 

\begin{defn}\label{def:type}
Let $\Sigma \subseteq \cl L_\diam$ be closed under subformulas and $\Phi^-,\Phi ^ + \subseteq \Sigma$. We say that the pair $\Phi = (\Phi ^- ; \Phi ^+) $ is a {\em two-sided $\Sigma$-type} if:
  \begin{enumerate}[label=(\alph*)]

  \item $\Phi^- \cap \Phi ^ +  = \varnothing$,
        \label{cond:type:intersection}
        
        \item $\Phi^- \cup \Phi ^ +  = \Sigma$,
         \label{cond:type:union}

  \item $\bot\not\in \Phi ^ +$,
        \label{cond:type:bot}

  \item if $\varphi \wedge \psi \in \Sigma$, then $\varphi\wedge\psi\in \Phi ^ +$ if and only if $\varphi,\psi\in \Phi^+$,
        \label{cond:type:posconj}


  \item if $\varphi \vee \psi \in \Sigma$, then $\varphi\vee\psi\in \Phi ^ +$ if and only if $\varphi \in \Phi^+$ or $\psi\in \Phi^+$,
        \label{cond:type:posdisj}


  \item if $\varphi\to\psi\in \Phi^+$, then either $\varphi \in \Phi^-$ or $\psi \in\Phi^+$, and
        \label{cond:type:implication}

\item\label{cond:type:diam} if $\diam \varphi \in \Phi^-$ then $\varphi \in \Phi^-$.

  \end{enumerate}
The set of two-sided $\Sigma$-types will be denoted $\type{\Sigma}$.
  
  We will write $\Phi \peqT \Psi$ if $\Phi^+ \subseteq \Psi^+$ (or, equivalently, if $\Psi^- \subseteq \Phi^-$). If $\Sigma \subseteq \Delta$ are both closed under subformulas, $\Phi \in \type \Sigma$ and $\Psi \in \type \Delta$, we will write $\Phi \subT \Psi $ if $\Phi^- \subseteq \Psi^-$ and $\Phi^+ \subseteq \Psi^+$.
\end{defn}

Often (but not always) we will want $\Sigma$ to be finite, in which case given $\Delta\subseteq \landi$ we write $\Sigma \Subset\Delta$ if $\Sigma$ is finite and closed under subformulas.
It is not hard to check that $\peqT$ is a partial order on $\type\Sigma$.
   Whenever $\Xi$ is an expression denoting a two-sided type, we write $\Xi^-$ and $\Xi^+$ to denote its components. Elements of $\type{\cl L_\diam}$ are {\em full types.}
Note that Fern\'andez-Duque \cite{FernandezITLc} uses one-sided types, but it is readily checked that a one-sided $\Sigma$-type $\Phi$ as defined there can be regarded as a two-sided type $\Psi$ by setting $\Psi^+=\Phi$ and $\Psi^- = \Sigma \setminus \Phi$.
Henceforth we will refer to two-sided types simply as {\em types.}

\subsection{Quasimodels}

Next we will define quasimodels; these are similar to models, except that valuations are replaced with a labelling function $\ell$. We first define the more basic notion of {\em $\Sigma$-labelled frame.}

\begin{defn}\label{frame}
Let $\Sigma \subseteq \cl L_\diam$ be closed under subformulas.  A {\em $\Sigma$-labelled frame} is a triple $\cl F= ( |\cl F|,{\peq}_\cl F,\ell_\cl F )$,
  where $\peq_\cl F$ is a partial order on $|\cl F|$
  and $\ell_\cl F\colon |\cl F| \to \type{\Sigma}$ is such that 
  \begin{enumerate}[label=(\alph*)]
    \item \label{cond:frame:monotony} whenever $w \peq_\cl F v$ it follows that $\ell_\cl F(w) \peqT \ell_\cl F(v)$, and
    \item \label{cond:frame:imp} whenever $\varphi\to\psi \in \ell_\cl F^-(w)$, there is $v \peq_\cl F w$ such that $\varphi\in \ell_\cl F^+(v)$ and $\psi \in \ell_\cl F^-(v)$.
  \end{enumerate}
 We say that $\cl F$ {\em falsifies} $\varphi \in \cl L_\diam$ if $\varphi \in \ell^-(w)$ for some $w \in W$.
\end{defn}

As before, we may omit the subindexes in $\peq_\cl F$, $S_\cl F$ and $\ell_\cl F$ when $\cl F$ is clear from context.
Labelled frames model only the intuitionistic aspect of the logic.
For the temporal dimension, let us define a new relation over types.

\begin{defn}\label{compatible}
Let $\Sigma \subseteq \cl L_\diam$ be closed under subformulas.
  We define a relation $\ST \subseteq \type{\Sigma} \times \type{\Sigma}$ by $\Phi \ST \Psi$ iff for all $\varphi \in \cl L$:
  \begin{enumerate}[label=(\alph*)]
  \item\label{ItCompOne} if $\tnext\varphi\in \Phi^+$ then $ \varphi\in \Psi^+$,
  \item\label{ItCompTwo} if $\tnext\varphi\in \Phi^-$ then $ \varphi\in \Psi^-$,
  \item\label{ItCompThree} if $\diam\varphi\in \Phi^+$ and $\varphi\in\Phi^-$ then $\diam\varphi\in \Psi^+$, and
  \item\label{ItCompFour} if $\diam\varphi\in \Phi^-$, then $\diam \varphi \in \Psi^-$.
  \end{enumerate}
\end{defn}

Quasimodels are then defined as labelled frames with a suitable binary relation.

\begin{defn} \label{def:quasimodel}
Given $\Sigma \subseteq \cl L_\diam$ closed under subformulas, a \emph{$\Sigma$-quasimodel} is a tuple $\cl Q = (|\cl Q|, \mathord{\peq}_\cl Q, S_\cl Q, \ell_\cl Q)$
  where $(|\cl Q|, \mathord{\peq}_\cl Q, \ell_\cl Q)$ is a labelled frame
  and $S_\cl Q$ is a binary relation over $|\cl Q|$ that is
  \begin{enumerate}
  
    \item\label{itSerial} {\em serial:} for all $w \in |\cl Q|$ there is $v \in |\cl Q| $ such that $w \mathrel S_\cl Q v$;
    \item {\em forward-confluent:} if $w \peq_\cl Q w'$ and $w \mathrel S_\cl Q v$, there is $v'$ such that $v\peq_\cl Q v'$ and $w'\mathrel S _\cl Q v'$;
    \item {\em sensible:} if $w \mathrel S_\cl Q v$ then $\ell _\cl Q (w) \ST \ell _\cl Q (v)$, and
    \item\label{itOmega} {\em $\omega$-sensible:} whenever $\diam\varphi\in \ell_\cl Q ^+(w)$, there are $n\geq 0$ and $v$ such that $w \mathrel S _\cl Q ^n v$ and $\varphi\in \ell _\cl Q^+(v)$.

  \end{enumerate}
A forward-confluent, sensible $\Sigma$-labelled frame is a {\em weak $\Sigma$-quasimodel,} and if $S_\cl Q$ is a function we say that $\cl Q$ is {\em deterministic.}
\end{defn}

\begin{figure}[h!]

\begin{center}

\begin{tikzpicture}[scale=.6]

\draw[thick] (0,0) circle (.35);

\draw (0,0) node {$w$};

\draw[very thick,->] (.5,0) -- (2.5,0);

\draw (1.5,.5) node {$S$};

\draw[thick] (3,0) circle (.35);

\draw (3,.02) node {$v$};

\draw[thick] (0,3) circle (.35);

\draw (.04,3-.03) node {$w'$};

\draw[very thick,->,dashed] (.5, 3) -- (2.5, 3);

\draw (1.5,3.5) node {$S$};

\draw[very thick,<-] (0, 2.5) -- (0, .5);

\draw (-.5,1.5) node {{\large$\rotatebox[origin=c]{90}{$\peq$}$}};

\draw[thick,dashed] (3, 3) circle (.35);

\draw (3+.04, 3-.03) node {$v'$};

\draw[very thick,<-,dashed] (3, 2.5) -- (3, .5);

\draw (3.5, 1.5) node {{\large$\rotatebox[origin=c]{90}{$\peq$}$}};

\end{tikzpicture}

\end{center}

\caption{If $S$ is forward-confluent, then the above diagram can always be completed.}
\end{figure}
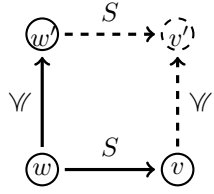

We may write {\em quasimodel} instead of $\Sigma$-quasimodel when $\Sigma$ is clear from context, and {\em full quasimodel} instead of {\em $\cl L_\diam$-quasimodel.} Similar conventions apply to labelled structures, weak quasimodels, etc.

\begin{defn}
	
Let  $\cl Q $ be a weak quasimodel and let $U$ be such that $U \subseteq |{\cl Q}|$. The restriction of ${\cl Q}$ with respect to $U$ is defined to be the structure
\[{\cl Q} \upharpoonright U = (|{\cl Q} \upharpoonright U|, \mathord{\peq}_{{\cl Q}\upharpoonright U}, S_{{\cl Q}\upharpoonright U}, \ell_{{\cl Q}\upharpoonright U}),\]
where:

\begin{multicols}2
\begin{enumerate}
\item $|{\cl Q} \upharpoonright U| = U$;
\item $\mathord{\peq}_{{\cl Q}\upharpoonright U}  = \mathord{\peq}_{\cl Q}\cap \left(U \times U\right)$;
\item $S_{{\cl Q}\upharpoonright U}  = S_{\cl Q} \cap \left(U \times U\right)$;
\item $\ell_{{\cl Q}\upharpoonright U} =  \ell_{\cl Q} \cap (U \times \type \Sigma) $.
\end{enumerate}
\end{multicols}
\end{defn}

\begin{lem}\label{LemmIsQuasi}
If $\cl Q$ is a weak quasimodel, $U \subseteq |\cl Q|$ is upward closed and $S_{\cl Q} \upharpoonright U$ is serial and $\omega$-sensible, then $\cl Q \upharpoonright U$ is a quasimodel.
\end{lem}
\proof We must show that ${\cl Q}$ satisfies all properties of Definition~\ref{def:quasimodel}. First we check that
\[(U , \mathord{\peq}_{{\cl Q}\upharpoonright U}, \ell_{{\cl Q}\upharpoonright U})\]
is a labelled frame.
The relation $\mathord{\peq}_{{\cl Q}\upharpoonright U}$ is a partial order, since restrictions of partial orders are partial orders.
Similarly, if $x \peq_{\cl Q\upharpoonright U} y$ it follows that
$x \peq_{{\cl Q}} y$, so that from the definition of
$\ell_{{\cl Q}\upharpoonright U}$ it is easy to deduce that
$\ell_{{\cl Q}\upharpoonright U}(x) \peqT \ell_{{\cl Q}\upharpoonright U}(y)$.

To check that condition~\ref{cond:frame:imp} holds, let us take $x \in U$ and a formula $\varphi \to \psi \in \ell_{{\cl Q}\upharpoonright U}^-(x)$. By definition, $\varphi \to \psi \in \ell_{{\cl Q}}^-(x)$ so there exists $y \in |{\cl Q}|$ such that $x \peq_{\cl Q} y$, $\varphi \in \ell_{\cl Q}^+(y)$ and $\psi \in \ell_{\cl Q}^-(y)$. Note that, since $U$ is upward closed then $y\in U$ and, by definition, $x  \peq_{{\cl Q}\upharpoonright U} y$, $\varphi \in \ell_{{\cl Q}\upharpoonright U}^+(y)$ and $\psi \in \ell_{{\cl Q}\upharpoonright U}^-(y)$, as needed.
\medskip

Now we check that the relation $S_{\cl Q\upharpoonright U}$ satisfies \eqref{itSerial}-\eqref{itOmega}.
Note that $S_{\cl Q\upharpoonright U}$ is serial and $\omega$-sensible by assumption and it is clearly sensible as $S_\cl Q$ was already sensible, so it remains to see that $S_{{\cl Q}\upharpoonright U}$ is forward-confluent. Take $x,y, z \in U$ such that $ x  \peq_{{\cl Q}\upharpoonright U} y$ and $x \mathrel S_{{\cl Q}\upharpoonright U} z$. By definition $x  \peq _{\cl Q} y$ and $x \mathrel S_{\cl Q} y$. Since $S_{\cl Q}$ is confluent, there exists $t \in |{\cl Q}|$ such that $ z  \peq _{\cl Q} t$ and $y \mathrel S_{\cl Q} t$. Since $U$ is upward closed $t \in U$ and, by definition, $y \mathrel S_{{\cl Q}\upharpoonright U} t$ and $z \mathord{\peq}_{{\cl Q}\upharpoonright U} t$.
\endproof

The following result of \cite{BoudouLICS} will be crucial for our completeness proof.

\begin{theorem}\label{TheoITLc}
A formula $\varphi\in \cl L_\diam$ is falsifiable over the class of expanding posets if and only if it is falsifiable over the class of finite, ${\rm sub}(\varphi)$-quasimodels.
\end{theorem}

As usual, if $\varphi$ is not derivable, we wish to produce an expanding poset where $\varphi$ is falsified, but in view of Theorem \ref{TheoITLc}, it suffices to falsify $\varphi$ on a quasimodel. This is convenient, as quasimodels are much easier to construct than models.

\section{The canonical model}\label{secCanMod}

The standard canonical model for $\logbasic$ it is only a full, weak, deterministic quasimodel rather than a proper model.
Nevertheless, it will be a useful ingredient in our completeness proof.
Since we are working over an intuitionistic logic, the role of maximal consistent sets will be played by prime types, which we define below; recall that {\em full types} are elements of $\type{\cl L_\diam}$. 

\begin{defn}\label{def:prime}
Given two sets of formulas $\Gamma$ and $\Delta$, we say that $\Delta$ is a consequence of $\Gamma$ (denoted by $\Gamma \vdash \Delta$) if there exist finite $\Gamma'\subseteq \Gamma$ and $\Delta' \subseteq \Delta$ such that $\logbasic \vdash \bigwedge \Gamma' \to \bigvee \Delta'$.

We say that a pair of sets $\Phi =(\Phi^-,\Phi^+)$ is {\em full} if $\Phi^-\cup \Phi^+ = \cl L_\diam$, and {\em consistent} if $\Phi^+ \not\vdash \Phi^-$. A full, consistent type is a {\em prime type.} The set of prime types will be denoted $\ptype$.
\end{defn}

Note that we are using the standard interpretation of $\Gamma \vdash \Delta$ in Gentzen-style calculi.
When working within a turnstyle, we will follow the usual proof-theoretic conventions of writing $\Gamma,\Delta$ instead of $\Gamma \cup \Delta$ and $\varphi$ instead of $\{\varphi\}$.
Observe that there is no clash in terminology regarding the use of the word {\em type:}

\begin{lem}\label{lemmPrimeIsType}
If $\Phi$ is a prime type then $\Phi$ is an $\cl L_\diam$-type.
\end{lem}

\begin{proof}
Let $\Phi$ be a prime type; we must check that $\Phi$ satisfies all conditions of Definition \ref{def:type}.
Condition \ref{cond:type:union} holds by assumption, and conditions \ref{cond:type:intersection} and \ref{cond:type:bot} follow from the consistency of $\Phi$.

The proofs of the other conditions are all similar to each other. For example, for \ref{cond:type:implication}, suppose that $\varphi \to \psi \in \Phi^+$ and $\varphi \not \in \Phi^-$. Since $\Phi$ is full, it follows that $\varphi \in \Phi^+$. But $\big (\varphi \wedge (\varphi \to \psi)\big) \to \psi$ is an intuitionistic tautology, so using the fact that $\Phi$ is consistent we see that $\psi \not \in \Psi^-$, which once again using condition \ref{cond:type:union} gives us $\psi \in \Phi^+$.
For condition \ref{cond:type:diam} we use \ref{ax10DiamFix}: if $\diam \varphi \in \Phi^-$ and $\varphi \in \Phi^+$ we would have that $\Phi$ is inconsistent, hence $\varphi \in \Phi^-$. The rest of the conditions are left to the reader.
\end{proof}

As with maximal consistent sets, prime types satisfy a Lindenbaum property.

\begin{lem}[Lindenbaum Lemma]\label{LemmLind}
	Let $\Gamma$ and $\Delta$ be sets of formulas. If $\Gamma \not\vdash \Delta$ then there exists a prime type $\Phi$ such that $\Gamma \subseteq \Phi^+$ and $\Delta \subseteq \Phi^-$.
\end{lem}

\proof
The proof is standard, but we provide a sketch.
Let $\varphi \in \cl L_\diam$. Note that either $\Gamma ,\varphi  \not\vdash \Delta $ or $\Gamma  \not \vdash \Delta,\varphi$, for otherwise by a cut rule (which is intuitionistically admissible) we would have $\Gamma \vdash \Delta$. Thus we can add $\varphi$ to $\Gamma \cup \Delta$, and by repeating this process for each element of $\cl L_\diam$ (or using Zorn's lemma) we can find suitable $\Phi$.
\endproof

Given a set $A$, let $\mathbb I_A$ denote the identity function on $A$. Then, the canonical model $\CMod$ is then defined as the labelled structure
\[\CMod = (|\CMod|,{\peq_\CIcon },S_\CIcon ,\ell_\CIcon ) \eqdef  (\type{\cl L_\diam},{\peq_T},S_T,{\mathbb I}_{\ptype})\upharpoonright \ptype;\]
in other words, $\CMod$ is the set of prime types with the usual ordering and successor relations. Note that $\ell_{\CIcon}$ is just the identity (i.e., $\ell_\CIcon (\Phi) = \Phi$).
We will usually omit writing $\ell_\CIcon $, as it has no effect on its argument.

\begin{prop}\label{prop:CisW}
The canonical model is a deterministic weak quasimodel.
\end{prop}

\proof
In view of Definition \ref{def:quasimodel}, we need
\begin{enumerate*}
\item $(|\CMod|,{\peq}_\CIcon ,\ell_\CIcon )$ to be a labelled frame,
\item $S_\CIcon $ to be a sensible forward-confluent function, and
\item $\ell_\CIcon $ to have $\type{\cl L_\diam}$ as its codomain.
\end{enumerate*}
The first item is Lemma \ref{lemm:normality} in the Appendix. That $S_\CIcon $ is a forward-confluent function is Lemma \ref{lemm:rcnext:prop}, and it is sensible since $\Phi \mathrel S_\CIcon  \Psi$ precisely when $\Phi \ST \Psi$. Finally, if $\Phi \in |\CMod|$ then $\ell_\CIcon (\Phi) = \Phi$, which is an element of $\type{\cl L_\diam}$ by Lemma \ref{lemmPrimeIsType}.
\endproof

\section{Simulations}\label{SecSim}

Simulations are relations between worlds in labelled spaces, and give rise to the appropriate notion of `substructure' for modal and intuitionistic logics.
We have used them to prove that a topological intuitionistic temporal logic has the finite quasimodel property \cite{FernandezITLc}, and they will also be useful for our completeness proof.
Below, recall that $\Phi \subseteq \Psi$ means that $\Phi ^ - \subseteq \Phi^-$  and $\Phi ^ + \subseteq \Phi^+$.

\begin{defn}
Let $\Sigma\subseteq \Delta \subseteq \landi$ be closed under subformulas, $\cl X$ be a $\Sigma$-labelled frame and $\cl Y$ be $\Delta$-labelled. A forward-confluent relation
${\simrel} \subseteq |{\cl X}|\times |{\cl Y}|$
is a {\em simulation} if, whenever $x\simrel y$, $\ell_{\cl X} (x) \subT  \ell_{\cl Y}(y).$
If there exists a simulation $\simrel$ such that $x \simrel y$, we write $(\cl X , x)\simu (\cl Y, y)$.

The relation $\simrel$ is a {\em dynamic simulation} between $\cl X$ and $\cl Y$ if ${
\mathrel S_\cl Y\simrel} \subseteq { \simrel \mathrel S_\cl X}
$.

\end{defn}

\begin{figure}[h!]

\begin{center}

\begin{tikzpicture}[scale=.6]

\draw[thick] (0,0) circle (.35);

\draw (.04,-.05) node {$y$};

\draw[very thick,->] (.5,0) -- (2.5,0);

\draw (1.5,-.5) node {$S_\cl Y$};

\draw[thick] (3,0) circle (.35);

\draw (3+.03,0) node {$y'$};

\draw[thick] (0,3) circle (.35);

\draw (.04,3-.04) node {$x$};

\draw[very thick,->,dashed] (.5,3) -- (2.5,3);

\draw (1.5,3.5) node {$S_\cl X$};

\draw[very thick,->] (0,2.5) -- (0,.5);

\draw (-.5,1.5) node {${\simrel}$};

\draw[thick,dashed] (3,3) circle (.35);

\draw (3+.03,3) node {$x'$};

\draw[very thick,->,dashed] (3,2.5) -- (3,.5);

\draw (3.5,1.5) node {${\simrel}$};

\end{tikzpicture}

\end{center}

\caption{If ${\simrel} \subseteq |\cl X| \times |\cl Y|$ is a dynamical simulation, this diagram can always be completed.}

\end{figure}
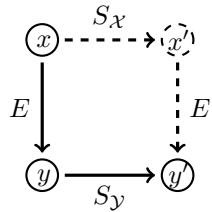

\ignore{
The following properties are readily verified:

\begin{lem}\label{LemmPropSim}

Let ${\cl X},{\cl Y},{\cl Z}$ be labelled systems and $\simrel\subseteq |{\cl X}|\times  |{\cl Y}|$, $\xi\subseteq |{\cl Y}|\times  |{\cl Z}|$ be simulations. Then:

\begin{enumerate}

\item\label{ItPropSimComp} $\xi\simrel\subseteq |{\cl X}|\times  |{\cl Z}|$ is a simulation. Moreover, if both $\simrel$ and $\xi$ are dynamic, then so is $\xi\simrel$.
\item \label{ItPropSinSub} If $U\subseteq |{\cl X}|$ and $V\subseteq |{\cl Y}|$ are open, then $\simrel\upharpoonright U\times V$ is a simulation.
\item\label{ItPropSimUn} If $\Xi\subseteq\mathcal P(|{\cl X}|\times  |{\cl Y}|)$ is a set of simulations, then $\bigcup \Xi$ is also a simulation.

\end{enumerate}

\end{lem}
}

The following is proven in \cite{FernandezITLc}. While the details of the construction given there are not important for our current purposes, the interested reader may find an overview in Appendix \ref{secFinFrm}.
Below, recall that $\Sigma \Subset \landi$ means that $\Sigma$ is finite and closed under subformulas.

\begin{theorem}\label{thmSurjI}
Given $\Sigma\Subset \cl L_\diam$, there exists a finite weak quasimodel $\irr\Sigma$
such that if $\cl A$ is any deterministic weak quasimodel then ${\rightbarharpoon} \subseteq |\irr\Sigma| \times |\cl A|$ is a surjective dynamic simulation.
\end{theorem}

Points of $\irr\Sigma$ are called {\em moments.}
One can think of $\irr\Sigma$ as a finite initial structure over the category of labelled weak quasimodels. 
Next, we will internalize the notion of simulating elements of $\irr\Sigma$ into the temporal language. This is achieved by the formulas $\Sim w$ given by the next proposition.

\begin{prop}\label{simulability}
Given $\Sigma \Subset   \landi$ and a finite $\Sigma$-labelled frame $\cl W$, there exist formulas $(\Sim w)_{w\in |\cl W|}$ such that for any fully labelled frame $\cl X$, $w\in |\cl W|$ and $x\in |\cl X|$, $\Sim w \in \ell^-(x)$ if and only if there is $y\seq x$ such that $(\cl W,w) \simu (\cl X, y)$.
\end{prop}

\proof
An explicit construction is given in Appendix \ref{secsubchar}.
\endproof

The next proposition allows us to emulate model-theoretic reasoning within $\landi$.



\begin{prop}\label{propsub}
Fix $\Sigma \Subset \landi$ and let $\cl I = \irr\Sigma $, $w\in |\irr{}|$ and $\psi\in \Sigma $.
\begin{enumerate}
\item\label{itPropsubOne}
If $\psi\in \ell^-({{w}})$, then $\vdash \psi \to \mathrm{Sim}({{w}}) $.

\item\label{itPropsubOneb} If $\psi\in \ell^+({w})$, then $\vdash \big (\psi \to \Sim {{w}} \big )\to \Sim w $.

\item\label{itPropsubThree} If ${{{w}}}\peq{{v}}$, then $\vdash \mathrm{Sim}({{{v}}})\to \mathrm{Sim}({{w}})$.

\item\label{itPropsubFour} $\vdash \displaystyle \bigwedge_{\psi \in \ell_{\irr{}}^- ({{w}})}  \mathrm{Sim}({{w}}) \rightarrow \psi.$

\item\label{itPropsubFive} $\vdash\displaystyle \tnext\bigwedge _{{{w}} \mathrel S_{\irr{}} {{{v}}}  }\mathrm{Sim}({{{v}}}) \to \mathrm{Sim}({{w}}).$
\end{enumerate}
\end{prop}

\proof

\noindent \eqref{itPropsubOne}
First assume that $\psi\in \ell^- ({w})$, and toward a contradiction that $\nvdash \psi \to \Sim{{w}}$. By the Lindenbaum lemma there is $\Gamma \in | \CMod | $ such that $\psi \to \Sim{{w}} \in \Gamma^-$. Thus for some $\Theta \seq_\CIcon \Gamma$ we have that $\psi \in \Theta^+$ and $\Sim{{w}} \in \Theta^-$. But then by Proposition \ref{simulability} we have that $(\cl W,{w}) \simu (\CMod,\Delta)$ for some $\Delta \seq_\CIcon \Theta$, so that $\psi\in \Delta^-$, and by upwards persistence $\psi \in \Theta^-$, contradicting the consistency of $\Theta$.
\medskip

\noindent \eqref{itPropsubOneb} If $\psi\in \ell^+ ({w})$, we proceed similarly. Assume toward a contradiction that $\nvdash \big ( \psi \to \Sim{{w}} \big ) \to \Sim w$. Then, reasoning as above there is $\Theta \in |\CMod|$ such that $\psi \to \Sim{{w}} \in \Theta^+$ and $\Sim w \in \Theta^-$. From Proposition \ref{simulability} we see that there is $\Delta \seq_c \Theta$ such that $(\cl W,w) \simu (\CMod,\Delta)$, so that $\psi \in \Delta^+$ and, once again by Proposition \ref{simulability}, $\Sim w \in \Delta^-$. It follows that $\psi \to \Sim w \not \in \Delta^+$; but in view of upward persistence, this contradicts that $\psi \to \Sim{{w}} \in \Theta^+$.

\medskip


\noindent \eqref{itPropsubThree} Suppose that ${v}\seq{w}$. Reasoning as above, it suffices to show that if $\Gamma \in |\CMod|$ is such that $\Sim{{w}} \in \Gamma ^-$, then also $\Sim{{v}} \in \Gamma ^-$. But if $\Sim{{w}} \in \Gamma ^-$, there is $\Theta \seq _\CIcon \Gamma$ such that $(\irr{},w) \simu (\CMod, \Theta)$. By forward confluence $(\irr{},v) \simu (\CMod, \Delta)$ for some $\Delta \seq_\CIcon \Theta$. Thus by Proposition \ref{simulability}, $\Sim {v} \in \Delta^-$ and by upwards persistence $\Sim{v} \in \Gamma^-$. Since $\Gamma \in |\CMod| $ was arbitrary, the claim follows.
\medskip

\noindent \eqref{itPropsubFour}
We prove that if $\Gamma \in |\CMod|$ is such that
\begin{equation}\label{EqGammaConjunc}
\bigwedge _{\psi\in \ell^- ({{w}})} \mathrm{Sim}({{w}})  \in \Gamma ^+,
\end{equation}
then $\psi \in \Gamma^+$. If \eqref{EqGammaConjunc} holds then by Theorem \ref{thmSurjI}, there is ${w} \in |\irr{}|$ with $(\irr{}, w) \simu (\CMod,\Gamma)$. By Proposition \ref{simulability}, $\Sim {w} \in \Gamma^-$, hence it follows from \eqref{EqGammaConjunc} that $\psi \not \in \ell^- ({{w}})$; but ${w}$ is $\Sigma$-typed and $\psi \in \Sigma$, so $\psi  \in \ell^+ ({{w}})$ and thus $\psi \in \Gamma^+$, as required.
\medskip

\noindent \eqref{itPropsubFive} Suppose that $\Gamma \in |\CMod|$ is such that
\begin{equation}\label{EqGammaCirc}
\tnext\bigwedge _{{{w}} \mathrel S_{\irr{}} {{{v}}}   }\Sim{{{v}}} \in \Gamma^+ ,
\end{equation}
and assume toward a contradiction that $\Sim{{w}} \in \Gamma^-$. Then $(\irr{}, w) \simu (\CMod,\Delta)$ for some $\Delta \seq_\CIcon \Gamma$. Since $\simu$ is a dynamic simulation, it follows that there is $v \in |\irr{}|$ with ${w} \mathrel S_{\irr{}} v$ and $(\irr{}, v) \simu  \big ( \CMod, S_\CIcon (\Delta) \big ) $. But $\Delta\seq_\CIcon \Gamma$, so that by Proposition \ref{simulability}, $\Sim v \in S^-_\CIcon (\Gamma)$, contradicting \eqref{EqGammaCirc}. 
\endproof

\section{The initial quasimodel}\label{seccan}

We are now ready to define our initial quasimodels. Given a finite set of formulas $\Sigma$, we will define a quasimodel $\cqm\Sigma$ falsifying all unprovable $\Sigma$-types. This quasimodel is a substructure of $\irr\Sigma$, containing only moments which are {\em possible} in the following sense:

\begin{defn}\label{defsound}
Fix $\Sigma\Subset \landi$. We say that a moment ${{w}} \in |\irr\Sigma|$ is {\em possible} if $\not \vdash \mathrm{Sim}({{w}})$, and denote the set of possible $\Sigma$-moments by $\unp\Sigma$.
\end{defn}

With this we are ready to define our initial structure, which as we will see later is indeed a quasimodel.

\begin{defn}
Given $\Sigma \Subset\landi$, we define the {\em initial structure} for $\Sigma$ by $\cqm \Sigma = \irr {\Sigma} \upharpoonright \unp\Sigma$.
\end{defn}

Our strategy from here on will be to show that canonical structures are indeed quasimodels; once we establish this, completeness of $\logbasic $ is an easy consequence. The most involved step will be showing that the successor relation on ${\cqm{\Sigma}}$ is $\omega$-sensible, but we begin with some simpler properties.

\begin{lem}\label{niceprop}
Let $\Sigma$ be a finite set of formulas, $\irr{} = \irr\Sigma$ and $\cqm{} = \cqm \Sigma$. Then, $|\cqm{}|$ is an upward-closed subset of $|\irr{}|$ and $S_{\cqm{}}$ is serial.
\end{lem}

\proof
To check that $|\cqm{}|$ is upward closed, let ${{w}}\in |\cqm{}|$ and suppose ${{{v}}}\seq{{w}}$. Now, by Proposition \ref{propsub}.\ref{itPropsubThree}, we have that
\[\vdash\mathrm{Sim}({{{v}}})\to \mathrm{Sim}({{w}});\]
hence if ${{w}}$ is possible, so is ${{{v}}}$.

To see that $S_{\cqm {} }$ is serial, observe that by Proposition \ref{propsub}.\ref{itPropsubFive}, if  ${{w}}\in |\cqm{}|$
for all ${{w}}\in|\irr{}|$,
\[\vdash \tnext\bigwedge_{{{w}}\mathrel S_{\irr{}} {{{v}}} }\mathrm{Sim}({{{v}}}) \to \mathrm{Sim}({{w}});\]
since ${{w}}$ is possible, it follows that for some ${{{v}}}$ with ${{w}} \mathrel S_{\irr{}}  {{{v}}}$, ${{{v}}}$ is possible as well, and thus ${{{v}}}\in|\cqm{}|.$
\endproof

\section{$\omega$-Sensibility}\label{secOmSens}

In this section we will show that $S_{\cqm{}}$ is $\omega$-sensible, the most difficult step in proving that $\cqm{}$ is a quasimodel. In other words, we must show that, given ${{w}}\in |\cqm{}|$ and $\diam \psi\in \ell^+({{w}})$, there is a finite path
\[{{w}}={{w}}_0 \mathrel S {{w}}_1 \mathrel S \hdots \mathrel S {{w}}_n,\]
where $\psi\in \ell^+ ({{w}}_n)$ and ${{w}}_i\in |\cqm{}|$ for all $i\leq n$.

\begin{defn} 
Let $\Sigma\Subset \landi$ and ${{w}},v\in \unp\Sigma$. Say that ${{{v}}}$ is {\em reachable} from ${{w}}$ if there is a finite path 
\[\overrightarrow {u}=\<u_0,...,u_n\>\]
of possible moments with $u_0={{w}}$, $u_n={{{v}}}$, and $u_i \mathrel S u_{i+1}$ for all $i<n$.
We denote the set of all possible moments that are reachable from ${{w}}$ by $R(w)$.
\end{defn}

\begin{lem}\label{syntactic}
If $\Sigma\Subset\landi$ and ${{w}}\in|\cqm\Sigma|$ then 
\[\vdash \tnext \bigwedge _{{{{v}}}\in R({{w}})}\mathrm{Sim}({{{v}}})\to \bigwedge_{{{{v}}}\in{R}({{w}})}\mathrm{Sim}({{{v}}}).\]
\end{lem}

\proof
Let $\irr{} = \irr\Sigma$. By Proposition \ref{propsub}.\ref{itPropsubFive} we have that, for all ${{{v}}}\in{R}({{w}})$,
\[\vdash \tnext\bigwedge_{{{{v}}} \mathrel S_{\irr{}} u }\mathrm{Sim}(u) \to \mathrm{Sim}({{{v}}}) .\]
Now, if $u\not\in\unp \Sigma$, then $\vdash \mathrm{Sim}(u)$, hence by \ref{ax14NecCirc} $\vdash \tnext \mathrm{Sim}(u)$, and we can remove $\Sim u$ from the conjunction using Lemma \ref{lemmReverse} and propositional reasoning.
Since ${{{v}}} \in R(w)$ was arbitrary, this shows that
\[\vdash \tnext \bigwedge_{{{{v}}}\in{R}({{w}})}\mathrm{Sim}({{{v}}})\to \bigwedge _{{{{v}}}\in{R}({{w}})}\mathrm{Sim}({{{v}}}).\]
\endproof

From this we obtain the following, which evidently implies $\omega$-sensibility:

\begin{prop}\label{tempinc}
If ${{w}}\in|\cqm\Sigma|$ and $\diam \psi\in \ell ({{w}})$, then there is ${{{v}}}\in{R}({{w}})$ such that $\psi\in \ell ({{{v}}})$.
\end{prop}

\proof Towards a contradiction, assume that ${{w}}\in \unp \Sigma$ and $\diam \psi\in \ell^+ ({{w}})$ but, for all ${{{v}}}\in{R}({{w}})$, $\psi \in \ell^-({{w}})$.

By Lemma \ref{syntactic},
\[\vdash \tnext \bigwedge _{{{{v}}}\in{R}({{w}})}\mathrm{Sim}({{{v}}})\to \bigwedge_{{{{v}}}\in{R}({{w}})}\mathrm{Sim}({{{v}}}).\]

But then we can use the $\diam$-induction rule \ref{ax12:ind:2} to show that
\[\vdash \diam \bigwedge _{{{{v}}}\in{R}({{w}})}\mathrm{Sim}({{{v}}})\to \bigwedge_{{{{v}}}\in{R}({{w}})}\mathrm{Sim}({{{v}}});\]
in particular,
\begin{equation}\label{other}
\vdash \diam \bigwedge _{{{{v}}}\in{R}({{w}})}\mathrm{Sim}({{{v}}})\to \mathrm{Sim}({{w}}).
\end{equation}

Now let ${{{v}}}\in{R}({{w}})$. By Proposition \ref{propsub}.\ref{itPropsubOne} and the assumption that $\psi \in \ell^-({{{v}}})$ we have that
\[\vdash \psi \to \mathrm{Sim}({{{v}}}) ,\]
and since ${{{v}}}$ was arbitrary,
\[\vdash \psi \to \bigwedge_{{{{v}}}\in {R}({{w}})}\mathrm{Sim}({{{v}}}) .\]
Using distributivity \ref{ax11:dist} we further have that
\[\vdash \diam \psi \rightarrow \diam  \bigwedge_{{{{v}}}\in {R}({{w}})}\mathrm{Sim}({{{v}}}).\]
This, along with (\ref{other}), shows that
\[\vdash \diam \psi \to \mathrm{Sim}({{w}});\]
however, by Proposition \ref{propsub}.\ref{itPropsubOneb} and our assumption that $\diam \psi\in \ell^+ ({{w}})$ we have that
\[\vdash \big ( \diam \psi \to \mathrm{Sim}({{w}}) \big ) \to \Sim w,\]
hence by modus ponens we obtain $\vdash \mathrm{Sim}({{w}}),$ which contradicts the assumption that ${{w}}\in \unp\Sigma$. We conclude that there can be no such ${{w}}$.
\endproof

\begin{cor}\label{laststretch}
Given any finite set of formulas $\Sigma$, $\cqm\Sigma$ is a quasimodel.
\end{cor}

\proof
Let $\cqm{} = \cqm \Sigma$. By Lemma \ref{niceprop}, $|\cqm{}|$ is upwards closed in $|\irr\Sigma|$ and $S_{\cqm{}}$ is serial, while by Proposition \ref{tempinc}, $S_{\cqm{}}$ is $\omega$-sensible. It follows from Lemma \ref{LemmIsQuasi} that $\cqm{}$ is a quasimodel.
\endproof

We are now ready to prove that $\logbasic $ is complete for the class of quasimodels.

\begin{theorem}\label{theocomp}
If $\varphi \in \landi$ is such that $\logbasic \not \vdash\varphi$, then $\varphi$ is falsifiable on a finite ${\rm sub}(\varphi)$-quasimodel.
\end{theorem}

\proof
We prove the contrapositive. Suppose $\varphi$ is an unprovable formula and let
\[W=\cbra {{w}}\in\irr{{\rm sub}(\varphi)}:\varphi\in \ell ({{w}})\cket.\]
Then, by Proposition \ref{propsub}.\ref{itPropsubFour} we have that
\[\vdash \bigwedge_{{w} \in W} \mathrm{Sim}({{w}}) \rightarrow \varphi;\]
since $\varphi$ is unprovable, it follows that some ${{w}}^\ast\in W$ is possible and hence ${{w}}^\ast\in \unp {{\rm sub}(\varphi)}$. By Corollary \ref{laststretch}, $\cqm {{\rm sub}(\varphi)}$ is a quasimodel.
\endproof

In view of Boudou et al.~\cite{BoudouLICS}, we immediately obtain completeness for the class of expanding posets.

\begin{cor}\label{corcomplete}
Given $\varphi\in \landi$, $\logbasic \vdash \varphi$ if and only if $\varphi$ is valid over the class of expanding posets.
\end{cor}

\proof
Soundness is Theorem \ref{ThmSoundZero} and completeness follows from Theorems \ref{theocomp} and
\ref{TheoITLc}.
\endproof

\section*{Concluding remarks}

We have provided a sound and complete axiomatization for the $\ubox$-free fragment of the expanding intuitionistic temporal logic $\itle$. With this we may develop syntactic techniques to decide validity over the class of expanding posets, complementing the semantic methods presented by Boudou and the authors \cite{BoudouCSL} and possibly leading to an elementary decision procedure.

Many questions remain open in this direction, perhaps most notably an extension to the full language with $\ubox$. This is likely to be a much more challenging problem, as the language with `henceforth' can distinguish between Kripke and topological models and hence methods based on non-deterministic quasimodels do not seem feasible.

The question of axiomatizing $\itlp$ (with persistent domains) is also of interest, but here it is possible that the logic is not even axiomatizable in principle. It may be that methods from products of modal logics \cite{mdml} can be employed here; for example, one can reduce tiling or related problems to show that certain products such as ${\sf LTL} \times {\sf S4}$ are not computably enumerable. However, even if such a reduction is possible, working over the more limited intuitionistic language poses an additional challenge.
Even computational lower bounds for these logics are not yet available, aside from the trivial {\sc pspace} bound obtained from the purely propositional fragment.

\bibliographystyle{plain}

\vfill

\pagebreak

\appendix

\section{Properties of the canonical model}\label{apCanMod}

The structure $\CMod$ is not a proper model, as it is not $\omega$-sensible. However, it comes quite close; it is a full, weak, deterministic quasimodel. We prove that it has all properties required by Definition \ref{def:quasimodel}.

\begin{lem}
\label{lemm:normality} $\M_\CIcon $ is a labelled frame.
\end{lem}
\begin{proof}
We know that $\peqT$ is a partial order and restrictions of partial orders are partial orders, so $\peq_\CIcon $ is a partial order. Moreover, $\ell_\CIcon $ is the identity, so $\Phi \peq_\CIcon  \Psi$ implies $\ell_\CIcon  (\Phi) \peqT \ell_\CIcon  (\Psi)$.

Now let $\Phi \in |\CMod|$ and assume that $\varphi \to \psi \in \Phi^-$. Note that $ \Phi^+,\varphi \not \vdash \psi$, for otherwise by intuitionistic reasoning we would have $\Phi^+\vdash \varphi \to \psi$, which is impossible if $\Phi$ is a prime type. By Lemma \ref{LemmLind}, there is a prime type $\Psi$ with $\Phi^+ \cup \{\varphi\} \subseteq \Psi^+$ and $\psi \in \Psi^-$. It follows that $\Phi \peq_\CIcon  \Psi$, $\varphi \in \Psi^+$ and $\psi \in \Psi^-$, as needed.
\end{proof}

\begin{lem}
\label{lemm:rcnext:prop} $S_{\CIcon }$ is a forward-confluent function.
\end{lem}
\begin{proof} 
For a set $\Gamma \subseteq \cl L_\diam$, recall that we have defined $\tnext \Gamma = \{\tnext \varphi :   \varphi \in \Gamma\}$. It will be convenient to introduce the notation
\[\remc \Gamma = \{\varphi : \tnext \varphi \in \Gamma\}.\]
With this, we show that $S_\CIcon$ is functional and forward-confluent.
\medskip

\noindent{\sc Functionality.} We claim that for all $\Phi, \Psi \in |\CMod|$,
\begin{equation}\label{EqSc}
\Phi \mathrel S_\CIcon \Psi \text{ if and only if }\Psi = (\remc \Phi^-,\remc\Phi^+). 
\end{equation}
We must check that $\Psi \in |\CMod|$. To see that $\Psi$ is full, let $\varphi \in \cl L_\diam$ be so that $\varphi \not \in \Psi^-$. It follows that $\tnext\varphi \not \in \Phi^-$, but $\Phi$ is full, so $\tnext\varphi \in \Phi^+$ and thus $\varphi \in \Psi^+$. Since $\varphi$ was arbitrary, $\Psi^-\cup \Psi^+ = \cl L_\diam$.

Next we check that $\Psi$ is consistent. If not, let $\Gamma \subseteq \Psi^+$ and $\Delta\subseteq \Psi^-$ be finite and such that $\bigwedge \Gamma \to \bigvee \Delta$ is derivable. Using \ref{ax14NecCirc} and \ref{ax05KNext} we see that $\tnext \bigwedge \Gamma \to  \tnext \bigvee \Delta$ is derivable, which in view of Lemma \ref{lemmReverse} implies that $ \bigwedge \tnext \Gamma \to  \bigvee  \tnext \Delta$ is derivable as well. But $\tnext \Gamma \subseteq \Phi^+$ and $\tnext \Delta \subseteq \Phi^-$, contradicting the fact that $\Phi$ is consistent.

Thus $\Psi \in |\CMod|$, and $\Phi \mathrel S_\CIcon \Psi$ holds provided that $\Phi \ST \Psi$. It is clear that clauses \ref{ItCompOne} and \ref{ItCompTwo} of Definition \ref{compatible} hold. If $\diam \varphi \in \Phi^+$ and $\varphi \not \in \Phi^+$, it follows that $\varphi \in \Phi^-$. By Lemma \ref{lemmReverse} $\diam\varphi \to \varphi \vee \tnext \diam \varphi$ is derivable, so we cannot have that $\tnext \diam \varphi \in \Phi^-$ and hence $\tnext \diam \varphi \in \Phi^+$, so that $\diam \varphi \in \Psi^+$. Similarly, if $\diam \varphi \in \Phi^-$ we have that $\tnext\diam\varphi \in \Phi^-$, for otherwise we obtain a contradiction from \ref{ax10DiamFix}. Therefore, $\diam\varphi \in \Psi^-$ as well.

To check that $\Psi$ is unique, suppose that $\Theta \in |\CMod|$ is such that $\Phi \mathrel S_\CIcon  \Theta$. Then if $\varphi \in \Psi^+$ it follows from \eqref{EqSc} that $\tnext \varphi \in \Phi^+$ and hence $\varphi \in \Theta^+$; by the same argument, if $\varphi \in \Psi^-$ it follows that $\varphi \in \Theta^-$, and hence $\Theta = \Psi$.
\medskip

\noindent {\sc Forward confluence:} Now that we have shown that $S_\CIcon$ is a function, we may treat it as such. Suppose that $\Phi \peq_\CIcon \Psi$; we must check that $S_\CIcon (\Phi) \peq_\CIcon S_\CIcon (\Psi)$. Let $\varphi \in S^+_\CIcon (\Phi)$. Using \eqref{EqSc}, we have that $\tnext\varphi \in \Phi^+$, hence $\tnext\varphi \in \Psi^+ $ and thus $\varphi \in S_\CIcon (\Psi^+)$. Since $\varphi \in S_\CIcon (\Phi)$ was arbitrary we obtain $S^+_\CIcon (\Phi) \peq_\CIcon S^+_\CIcon (\Psi)$, as needed.
\end{proof}

\section{Simulation formulas}\label{secsubchar}

In this appendix, we show that there exist $\landi$ formulas defining points in finite frames up to simulability, i.e.~that if $\cl W$ is a finite frame and $w\in |\cl W|$, there exists a formula $\Sim w$ such that for all labelled frames $\cl M$ and all $x\in |\cl M|$, $\cl M,x  \models x$ if and only if $(\cl W,w) \simu (\cl M,x)$.
In contrast, such formulas do not exist in the classical modal language for finite $\sf S4$ models \cite{FernandezSimulability}, but they {\em can} be constructed using a polyadic `tangled' modality. This tangled modalilty was proven to be expressively equivalent to the $\mu$-calculus over transitive frames by Dawar and Otto \cite{do}, and later axiomatized for several classes of models by Fern\'andez-Duque \cite{FernandezTangle} and Goldblatt and Hodkinson \cite{GoldblattH17,GoldblattTangleSL}.

Simulation formulas were used in \cite{dtlaxiom} to provide a sound and complete axiomatization of {\em dynamic topological logic} \cite{arte,kmints}, a classical tri-modal system closely related to $\itle$, where the intuitionistic implication is replaced by an $\sf S4$ modality.
One can use the fact that simulability is not definable over the modal language to prove that the natural axiomatization suggested by Kremer and Mints \cite{kmints} of dynamic topological logic was incomplete for its topological, let alone its Kripke, semantics \cite{FernandezNonFin}.

While simulability is not modally definable, it is definable over the language of intuitionistic logic, as finite frames \cite{JonghY09} (and hence models) are already definable up to simulation in the intuitionistic language. This may be surprising, as the intuitionistic language is less expressive than the modal language; however, intuitionistic models are posets rather than arbitrary preorders, and this allows us to define simulability formulas by recusion on $\prec$.

\begin{defn}
Fix $\Sigma\Subset \landi$ and let $\cl W$ be a finite $\Sigma$-labeled frame. Given $w\in |\cl W|$, we define a formula $\Sim w$ by backwards induction on ${\peq} = {\peq_\cl W}$ by
\[\Sim w = \bigwedge \ell^+(w) \rightarrow \bigvee \ell^-(w) \vee \bigvee_{v\succ w} \Sim v .\]
\end{defn}

\begin{prop}\label{simulabilitydef}
Given $\Sigma \Subset \Delta \subseteq \landi$, a finite $\Sigma$-labelled frame $\cl W$, a $\Delta$-labelled frame $\cl X$ and $w\in |\cl W|$, $x \in |\cl X |$:
\begin{enumerate}

\item\label{simulability:c1} if $\Sim w \in \ell_\cl X^-(x)$ then there is $y\seq x$ such that $(\cl W,w) \simu (\cl X, y)$, and

\item\label{simulability:c2} if there is $y\seq x$ such that $ (\cl W,w) \simu (\cl X, y) $ then $\Sim w \not \in \ell_\cl X^+(x)$.

\end{enumerate}
\end{prop}

\proof Each claim is proved by backward induction on $\peq$.
\medskip

\noindent \eqref{simulability:c1} Let us first consider the base case, when there is no $v \succ w$. Assume that $\Sim w \in \ell^-(x)$. From the definition of labelled frame $\bigwedge \ell_\cl W^+(w) \in \ell_\cl X^+(y)$ and  $\bigvee \ell_\cl W^-(w) \in \ell_\cl X^-(y)$ for some $y \seq x$. From the definition of type it follows that $\ell_\cl W^+(w) \subseteq \ell_\cl X^+(y)$ and $\ell_\cl W^-(w) \subseteq \ell_\cl X^-(y)$, so that $\ell_\cl W(w) \subT \ell_\cl X(y)$. It follows that ${\simrel} \eqdef \lbrace (w, y)\rbrace$ is a simulation, so $(\cl W, w) \simu  ({\cl X}, y )$.

For the inductive step, let us assume that the lemma is proved for all $v \succ w$.
Assume that $\Sim w \in \ell_\cl X^-(x)$. From Condition~\ref{cond:frame:imp} it follows that $\bigwedge \ell_\cl W^+(w) \in \ell_\cl X^+(y)$,  $\bigvee \ell_\cl W^-(w) \in \ell_\cl X^-(y)$ and $\bigvee _{v \prec w} \Sim v \in \ell_\cl X^-(y)$ for some $y \seq x$. By following a similar reasoning as in the base case we can conclude that $\ell_\cl W(w) \subseteq  \ell_\cl X(y)$, and moreover, that $\Sim v \in \ell_\cl X^-(y)$ for all $v \succ w$. By induction hypothesis we conclude that for all $v \succ w$, there exists a simulation $\simrel_{v}$ such that $v \simrel_{v} z_{v}$ for some $z_{v} \seq y$. Let
\[{\simrel} \eqdef \lbrace (w, y) \rbrace \cup \bigcup \limits_{v \succ w} \simrel_{v} .\]
The reader may check that $\simrel$ is a simulation and that $w \simrel y \seq x$, so that $(\cl W,w) \simu ({\cl X}, y)$, as needed.
\medskip

\noindent \eqref{simulability:c2}
For the base case, assume that $(\cl W,w) \simu (\cl X, y)$ for some $y \seq x$, so there exists a simulation $\simrel$ such that $w \simrel y  $. It follows that $\ell_\cl W^+(w) \subseteq \ell_\cl X^+(y)$ and $\ell_\cl W^-(w) \subseteq \ell_\cl X^-(y)$. From conditions \ref{cond:type:posconj} and \ref{cond:type:posdisj} of the definition of type (Definition \ref{def:type}), it follows that $\bigwedge \ell_\cl W^+(w) \not \in \ell_\cl X^-(y)$ and $\bigvee \ell_\cl W^-(w) \not \in \ell_\cl X^+(y)$. But then, condition \ref{cond:type:implication} gives us $\Sim w \not \in \ell_\cl X^+(y)$, so $\Sim w \not \in \ell_\cl X^+(x)$.

For the inductive step, by the same reasoning as in the base case it follows that $\bigwedge \ell_\cl W^+(w) \not \in \ell_\cl X^-(y)$ and $\bigvee \ell_\cl W^-(w) \not \in \ell_\cl X^+(y)$. Now, let $v$ be such that $v \succ w$. Since $\simrel$ is forward confluent then $v \simrel z_{v}$ for some $z_{v} \seq y$. By induction hypothesis, $\Sim v \not \in \ell^+(z_{v})$, so $\Sim v \not \in \ell^+(y)$. Since $v$ was arbitrary we conclude that $\bigvee _{v \succ w} \Sim v \not \in \ell^+(y)$. Finally, from condition \ref{cond:type:implication} of Definition \ref{def:type} and the fact that $y \peq x$ we get that $\Sim w \not \in \ell^+(x)$. 	
\endproof

\section{The finite initial frame}\label{secFinFrm}

In this appendix we review the construction of the structure $\irr\Sigma$ of Theorem \ref{thmSurjI}. The worlds of this structure are called {\em irreducible $\Sigma$-moments.} The intuition is that a $\Sigma$-moment represents all the information that holds at the same `moment of time'. Recall that we write $\Sigma\Subset\cl L_\diam$ if $\Sigma\subseteq \cl L_\diam$ is finite and closed under subformulas. We omit all proofs, which can be found in \cite{FernandezITLc}.

\begin{defn}
Let $\Sigma \Subset \cl L_\diam$. A {\em $\Sigma$-moment} is a $\Sigma$-labelled space ${{\fr w}}$ such that $ \left ( |{{\fr w}}|,\peq_{\fr w} \right )$ is a finite tree with unique root $\Root{{{\fr w}}}$.
\end{defn}

Note that moments can be arbitrarily large. In order to obtain a finite structure we will restrict the set of moments to those that are, in a sense, no bigger than they need to be. To be precise, we want them to be minimal with respect to $\redu$, which we define below.

\begin{defn}
Let $\Sigma\Subset\landi$ and ${{\fr w}},{{{\fr v}}}$ be $\Sigma$-moments. We write 

\begin{enumerate}

\item ${{\fr w}}\sqsubseteq {{{\fr v}}}$ if $|{{\fr w}}|\subseteq |{{{\fr v}}}|$, ${\peq_{{\fr w}}}={\peq_{{{\fr v}}}\upharpoonright |{{\fr w}}|}$, and $\ell_{{\fr w}} =\ell_{{{\fr v}}}\upharpoonright |{{\fr w}}|$;

\item ${{\fr w}}\redu{{{\fr v}}}$ if if ${{\fr w}}\sqsubseteq{{{\fr v}}}$ and there is a forward confluent, surjective function $\pi\colon |{{{\fr v}}}|\to |{{\fr w}}|$ such that $\ell_{{{\fr v}}}(v)=\ell_{{\fr w}}(\pi(v))$ for all $v\in |{{{\fr v}}}|$ and $\pi^2=\pi$. We say that ${{\fr w}}$ is a {\em reduct} of ${{{\fr v}}}$ and $\pi$ is a {\em reduction}.

\end{enumerate}

\end{defn} 

Note that the condition $\pi^2=\pi$ is equivalent to requiring $\pi(w)=w$ whenever $w\in |{\fw}|$.
Irreducible moments are the minimal moments under $\redu$. 

\begin{defn}
Let $\Sigma \Subset \landi$. A $\Sigma$-moment ${\fw}$ is {\em irreducible} if whenever ${\fw} \redu\fv$, it follows that ${\fw}=\fv$. The set of irreducible moments is denoted $I_\Sigma$.
\end{defn}

To view $I_\Sigma$ as a labeled frame, we need to equip it with a suitable partial order.

\begin{defn}
Let ${{\fr w}} \in I_\Sigma$. For $w\in |{{\fr w}}|$, let ${{\fr w}} [ w ] = {{\fw}\upharpoonright \mathord\uparrow w}$, i.e.,
\[{{\fr w}} [ w ] = \big ( \ \mathord\uparrow w \ , \ {\peq_{{{\fr w}}}} \upharpoonright \mathord\uparrow w  \ , \ \ell_{{{\fr w}}}   \upharpoonright \mathord\uparrow w \ \big ) .\]
We write ${{\fv}} \leq {{{\fw}}}$ if ${{{\fr v}}}={{\fr w}} [ w ] $ for some $w\in |{{\fr w}}|$.
\end{defn}

It is shown in \cite{FernandezITLc} that if ${\fw}$ is irreducible and $\fv \leq {\fw}$, $\fv$ is irreducible as well.
To obtain a weak quasimodel, it remains to define a sensible relation on $I_\Sigma$.

\begin{defn}\label{ts}
If $\Sigma \Subset \landi$ and $\fr w, \fr v \in I_\Sigma$, we define ${{{\fr v}}} \mapsto  {{\fr w}}$ if there exists a sensible, forward-confluent relation
$S\subseteq |{{{\fr v}}}|\times |{{\fr w}}|$
such that $\Root{{{{\fr v}}}}\mathrel S \Root{{{\fr w}}}$.
\end{defn}

We are now ready to define our initial weak quasimodel.

\begin{defn}
Given $\Sigma\Subset\landi$, we define $\irr{} = \irr\Sigma$ to be the structure
$ \< |\irr{}|,{\peq_{\irr{}}},S_{\irr{}},\ell_{\irr{}}\>$, where $|\irr{}| = I_\Sigma$, $\fv \peq_{\irr{}} \fw$ if and only if $\fv \geq \fw$, $\fw \mathrel S_{\irr{}} \fv$ if and only if $\fw \mapsto \fv $, and $\ell_{\irr{}}(\fw) = \ell_\fw (r_\fw)$.
\end{defn}

Note that in this construction, the moments accessible from $\fw$ are {\em smaller} than $\fw$, and thus we use the reverse partial order to interpret implication.
The structure $\irr\Sigma$ is always finite, a fact that is used in an essential way in our completeness proof. Below, $2^n_m$ denotes the superexponential function.

\begin{theorem}\label{thmSurjIrr}
Let $\Sigma\Subset\landi$ and let $s = \# \Sigma$. Then, $\irr\Sigma$ is a weak $\Sigma$-quasimodel and $\# I_\Sigma \leq 2^{s^2+s}_{s+1}$.
Moreover, if $\Sigma\Subset \landi$ and $\cl A$ is any deterministic weak quasimodel then ${\simu} \subseteq I_\Sigma \times |\cl A|$ is a surjective dynamic simulation.
\end{theorem}

In fact, the claim proven in Fern\'andez-Duque \cite{FernandezITLc} is more general in that $\cl A$ may belong to a wider class of topological weak qusimodels, but this special case will suffice for our purposes.

\forrefs{
\section{Status of unpublished material}\label{appReferees}

We include this short section for the benefit of the referees; it will be removed from the final version.
Our main result relies heavily on Theorem \ref{TheoITLc}, proven in Boudou et al.~\cite{BoudouLICS}, and Theorem \ref{thmSurjI}, proven in Fern\'andez-Duque \cite{FernandezITLc}. Both of these works are unpublished, but full versions are available on {\em arXiv;} precise references are provided in the bibliography.

Fern\'andez-Duque \cite{FernandezITLc} was submitted to {\em Logical Methods in Computer Science} in April 2017. After a long wait we have recently received a positive review which recommended only minor revisions. It should be noted that Theorem \ref{thmSurjIrr} is stated there for the case when $\cl A$ is a model, but the proof applies to deterministic weak quasimodels without modification and the statement will be updated in the final version.

Boudou et al.~\cite{BoudouLICS} has been submitted to {\em Logics in Computer Science} this year. It received generally positive reviews which identified only minor typos, although the final decision regarding acceptance will not be communicated until March 31st.
For neither of \cite{BoudouLICS,FernandezITLc} did any of the referees express doubts about the correctness of the results.
}

\end{document}